\documentclass[12pt,reqno]{amsart}
\usepackage[usenames]{color}
\usepackage{amsmath,verbatim,graphicx,epstopdf,enumerate}
\usepackage[usenames]{color}
\usepackage{soul}
\newcommand{\mb}{\mathbb}
\newcommand{\mc}{\mathcal}
\newcommand{\m}{\mbox}
\newcommand{\f}{\frac}
\newcommand{\ld}{\lambda}
\newcommand{\og}{\omega}

\newcommand{\rt}{\sqrt}

\newcommand{\dd}{\partial}

\newcommand{\nn}{\nonumber}
\newcommand{\ep}{\epsilon}
\newcommand{\qd}{\quad}

\newcommand{\Og}{\Omega}

\newcommand{\Ld}{\Lambda}

\newcommand{\iy}{\infty}

\newcommand{\tn}{\textnormal}
\newcommand{\del}{\partial}
\usepackage[pagewise]{lineno}
\newtheorem{theorem}{Theorem}[section]
\newtheorem{lemma}[theorem]{Lemma}
\newtheorem{proposition}[theorem]{Proposition}
\newtheorem{corollary}[theorem]{Corollary}

\numberwithin{equation}{section}

\newcommand{\la}{\mathcal{L}}

\newcommand{\bp}{\begin{prob}}
	\newcommand{\bpr}{\begin{proof}}
		\newcommand{\epr}{\end{proof}}

	\newcommand{\lb}{\left(}

	\newcommand{\rb}{\right)}
	\renewcommand{\d}{\mathrm{d}}

	\newcommand{\Rb}{\mathbb{R}}
	\newcommand{\Sb}{\mathbb{S}}

	\newcommand{\D}{\mathrm{d}}

	\usepackage{amsmath}
	\theoremstyle{definition}

	\title[Partial data inverse problem]{Stability estimate for a partial data inverse problem for the convection-diffusion equation}
	
	\author[Senapati and Vashisth]{Soumen Senapati$^{\dagger}$ and Manmohan Vashisth$^{*}$}
	\address{$^{\dagger}$ TIFR Centre for Applicable Mathematics, Bangalore 560065, India. 
		\newline\indent E-mail:{\tt \ soumen@tifrbng.res.in}}
	\address{$^{*}$ Department of Mathematics, Indian Institute of Technology, Jammu 181221, India.
		\newline
		\indent E-mail:{\tt\  manmohan.vashisth@iitjammu.ac.in, \ manmohanvashisth@gmail.com}}
	
\begin{document}		
		\maketitle
		
		\begin{abstract}
			In this article, we study the stability in the inverse problem of determining the time-dependent convection term and density coefficient appearing  in the convection-diffusion equation, from partial boundary measurements. For dimension $n\ge 2$, we show the convection term (modulo the gauge term) admits log-log stability, whereas log-log-log stability estimate is obtained for the density coefficient.  
		\end{abstract}
		
		\vspace{2mm}
		\textbf{Keywords:} Inverse problems, partial Dirichlet to Neumann map, parabolic equation, Carleman estimates, stability estimate.\\
		
		\textbf{Mathematics subject classification 2010:} 35R30, 35K20.
		\section{Introduction}\label{Introduction}
		For $n\ge 2$, let $\Omega\subset\mathbb{R}^{n}$ be a bounded simply connected domain having $C^2$ smooth boundary $\Gamma=\partial\Omega$. For $T>0$, let us introduce the parabolic operator $\mc{L}_{A,q}$ in the cylinder $Q:=(0,T)\times\Omega$ by 
		\begin{align}\label{Definition of LAq}
		\mc{L}_{A,q}:=\partial_{t}-\sum_{j=1}^{n}\left(\partial_{j}+{A_{j}(t,x)}\right)^{2}+q(t,x),
		\end{align}
		where $A(t,x):=\lb A_1(t,x),A_2(t,x),...,A_n(t,x)\rb\in(W^{1,\iy}(Q))^n$ and $q\in L^{\iy}(Q)$. We consider an initial boundary value problem (IBVP) known for the convection-diffusion equation which models physical processes like mass or heat transfer within a body and, also appears in probabilistic study of diffusion process (like, the Fokker-Planck and Kolmogorov equations), finance (like, the Black–Scholes or the Ornstein-Uhlenbeck processes) and chemical engineering (for describing the movement of macro-particles)
		\begin{align}\label{definition of operator}
		\begin{aligned}
		\begin{cases}
		\mc{L}_{A,q} u(t,x)=0,\ \ (t,x)\in Q,\\
		u(0,x)=0,\ \ x\in\Omega,\\
		u(t,x)=f(t,x), \ \ (t,x)\in\Sigma:=(0,T)\times\Gamma.
		\end{cases}
		\end{aligned}
		\end{align}
		
		We first briefly discuss some well-posedness results regarding the above IBVP. Following \cite{Caro_Kian_Convection_nonlinear}, one can consider suitable spaces for the forward problem \eqref{definition of operator}. However, in the context of our article, we assume more regularity on the coefficients in the operator \eqref{Definition of LAq}. Therefore, we define the admissible set of coefficients for a given $m>0$ by 
		\begin{align}
		\mc{M}(m)=\big\{(A,q)\in (W^{2,\iy}(Q))^n\times W^{1,\iy}(Q);\|A\|_{W^{2,\iy}(Q)}+\|q\|_{W^{1,\iy}(Q)}\le m\big\}.\nn
		\end{align} 
		Now, we introduce some time-dependent Sobolev spaces for $p,q$ being non-negative real numbers and $M=\Omega$ or, $\Gamma$
		\begin{align*}
		H^{p,q}\left((0,T)\times M\right):=L^2\left(0,T;H^p(M)\right)\cap H^q\left(0,T;L^2(M)\right),
		\end{align*}
		equipped with the norm
		\begin{align*}
		\|u\|_{H^{p,q}\left((0,T)\times M\right)}=\|u\|_{L^2\left(0,T;H^p(M)\right)}+\|u\|_{H^q\left(0,T;L^2(M)\right)}.
		\end{align*}
		Further, we denote $H^{p,q}_0(\Sigma):=\{f\in H^{p,q}(\Sigma);f(0,x)=0,\ x\in\Omega\}$. 
		The existence of unique solution $u\in H^{2,1}(Q)$ to the IBVP \eqref{definition of operator} for a given Dirichlet data $f\in H_0^{\f{3}{2},\f{3}{4}}(\Sigma)$ follows from \cite{Lion-Magenes2}. Also then, we have some $C>0$ depending only on $m$ and $Q$ such that
		\begin{align*}
		\|u\|_{H^{2,1}(Q)}\le C\|f\|_{H_0^{\f{3}{2},\f{3}{4}}(\Sigma)}.
		\end{align*}
		We can define the Dirichlet to Neumann (DN) map $\Lambda^{*}_{A,q}:H_0^{\f{3}{2},\f{3}{4}}(\Sigma)\to H^{\f{1}{2},\f{1}{4}}(\Sigma)$ as 
		\begin{align*}
		\Lambda^{*}_{A,q}(f)=\lb \partial_{\nu}u+2(\nu\cdot A)u\rb|_{\Sigma},
		\end{align*}
		where $u$ solves \eqref{definition of operator} and $\nu(x)$ denotes the unit outward normal at $x\in\Gamma$. 
		
		The inverse problem under consideration here is the stable recovery of time-evolving properties of a homogeneous medium such as $A$ and $q$, by applying heat source on $\Sigma$ and measuring the heat flux on a part of $\Sigma$. To be precise, we study the stability aspects for unique recovery of $(A,q)$ from a partial DN map which measures the Neumann outputs on a part of $\Sigma$ related to a small open neighborhood of the $\og_0$-illuminated face which is defined in \eqref{boundary portions}. It is known that the convection term can be recovered only under the divergence free condition (with respect to space variables) because of the non-uniqueness associated to the gauge transform (see \cite{Pohjola_steady_state_CD,Sahoo_Vashisth}). 
		We derive a stability estimate for the divergence free convection term $A$. In doing so, Vessella's conditional stability result \cite{Vessella} will be used crucially. Also we borrow an important construction for the principal term in the geometric optics solutions from \cite{Kian_Soccorsi_Holder_stability_Scrodinger} which was originally used in the framework of dynamical Schr\"odinger equation. The decay in the remainder terms of the geometric optics solutions follows from a Carleman estimate. For stability of the density coefficient, we again use Vessella's result \cite{Vessella} in combination with the stability estimate of $A$. 
		
		The issues regarding unique and stable determination of coefficients appearing in parabolic PDEs from boundary measurements have attracted much attention during last several decades. Motivated by the seminal work \cite{Sylvester_Uhlmann_Calderon_problem_1987} by Sylvester and Uhlmann, Isakov in \cite{Isakov_Completeness} uniquely determined the time-dependent coefficient when $A=0$, by using an argument based on completeness of the product of solutions. The stability issues of the same problem has been resolved by Choulli in \cite{Choulli_Book}. In \cite{Avdonin_Seidman_parabolic}, Avdonin and Seidman used boundary control (BC) method pioneered by Belishev which is further developed by Katchalov, Kurylev and Lassas (see \cite{Belishev_BC_Method_Recent_Progress,Katchalov_Kurylev_Lassas_Book_2001} and references therein), to establish uniqueness result for time-independent $q$. In the absence of any zeroth order term, Cheng and Yamamoto proved in \cite{Cheng_Yamamoto_Global_Convection_2D_DN,Chen_Yamamoto_parabolic,Cheng_Yamamoto_Steady} uniqueness of convection term which belongs to some Lebesgue spaces from single measurement when $n=2$. Gaitan and Kian in \cite{Gaitan_Kian_Cylinderical_Stability} obtained stable determination result for time-dependent $q$ in a bounded cylindrical domain when $A=0$ which was further generalized in the article \cite{Kian_Yamamoto} by Kian and Yamamoto proving analogous results in time-fractional diffusion equation settings. In \cite{Choulli_Kian_parabolic_partial}, Choulli and Kian derived logarithmic stability estimates for time-dependent term $q$ working only with partial DN map, in the absence of first order coefficients. In \cite{Bellassoued_Rassas}, Bellassoued and Rassas stably determined the convection term $A$ and density coefficient $q$ both of which are time-independent. Vashisth and Sahoo in \cite{Sahoo_Vashisth} obtained unique determination result for time-dependent convection term (modulo gauge equivalence) and density coefficient from full Dirichlet and partial Neumann data.  In this work, we have proved the stability estimate for determining the time-dependent convection term and the density coefficients from the knowledge of full Dirichlet data and the Neumann data measured on a portion which is slightly bigger than half of the lateral boundary.   We refer to \cite{Bellassoued_Kian_Soccorsi_parabolic_single_measurement,Choulli_Abstract_IP,Choulli_Abstract_IP_Applications,Choulli_Kian_parabolic_Parabolic,Choulli_Kian_Soccorsi_waveguide,Choulli_Book,Deng_Yu_Yang_First_order_parabolic_1995,Gaitan_Kian_Cylinderical_Stability,Isakov_semilinear,Isakov_Book,Nakamura_Sasayama_Parabolic} for more works in inverse problems related to parabolic PDEs. Also, there have been a considerable amount of work done in the context of hyperbolic and dynamical Schr\"odinger equations (see \cite{Avdonin and Belishev,BP,Bellassoued_Ferreira_anisotropic_Schrodinger,Bellassoued_Kian_Soccorsi_Scrodinger_infinite,Ibtissem_Magnetic_Scrodinger_time,Bukhgeuim_Klibanov_hyperbolic,Eskin_electromagentic_time,Kian_Phan_Socorsi_Carleman_Infinite_Cylinder,Kian_Phan_Socorsi_Unbounded,Kian_Soccorsi_Holder_stability_Scrodinger,Kian_Tetlow_Holder_stability_dynamical Scrodinger,Ibtissem_wave_equation,Rakesh_Symes_Uniqueness_1988,Salazar,Bellassoued_Jellali_Yamamoto_Lipschitz_stability_hyperbolic,Bellassoued_Jellali_Yamamoto_stability_hyperbolic,Hu_Kian_Wave_equation, Kian_ptential_uniqueness_wave,Kian_potential_stability_wave,Kian_damping,Kian-Oksanen,Krishnan_Vashisth_Relativistic,Senapati_stability_hyperbolic,Mishra_Vashisth_stability,Kian_stability_infinite wave guide wave} and references therein).     
		
		The article is organized as follows. In \S 2, the main result of the article is stated. Then boundary and interior Carleman estimates for the operator $\mc{L}_{A,q}$ blue are derived in \S 3, followed by the construction of geometric optics solutions in \S 4. Finally, we discuss the stable determination results for the convection term $A$ and density coefficient $q$ in \S 5.

		\section{Statement of the main result}\label{Main result statement}
		We begin this section by introducing some notations. Following \cite{Bukhgeim_Uhlmann_Calderon_problem_partial_Cauchy_data_2002},  
		fix an $\omega_{0}\in\mathbb{S}^{n-1}$ and define the $\omega_{0}$-shadowed and $\omega_{0}$-illuminated faces by 
		\begin{align*}
		\partial\Omega_{+}(\omega_{0}):=\left\{x\in\partial\Omega:\ \nu(x)\cdot\omega_{0}\geq 0 \right\},\ \ \partial\Omega_{-}(\omega_{0}):=\left\{x\in\partial\Omega:\ \nu(x)\cdot\omega_{0}\leq 0 \right\}
		\end{align*}
		of $\partial\Omega$ where $\nu(x)$ is the outward unit normal to $\partial\Omega$ at $x\in\partial\Omega$. For a given small $\ep>0$, we define the small open neighborhoods of $\partial\Omega_{+}(\omega_{0})$ and  $\partial\Omega_{-}(\omega_{0})$ by  
		\begin{align}\label{boundary portions}
		\del\Omega_{+,{\epsilon}/{2}}(\omega_{0}):=\left \{x\in\del\Omega;\ \nu(x)\cdot\omega_{0}>\f{\epsilon}{2}\right\},\qd\mbox{and}\qd \del\Omega_{-,{\epsilon}/{2}}(\omega_{0}):=\left\{x\in\del\Omega;\ \nu(x)\cdot\omega_{0}<\f{\epsilon}{2}\right\}.
		\end{align}
		respectively. 
		Corresponding to $\partial\Omega_{\pm}(\omega_{0})$ and $\del\Omega_{\pm,{\epsilon}/{2}}(\omega_0)$, we denote the lateral boundary parts by 
		$\Sigma_{\pm}(\omega_{0}):=(0,T)\times\partial\Omega_{\pm}(\omega_{0})$ and $\Sigma_{\pm,{\epsilon}/{2}}(\omega_0):=(0,T)\times\del\Omega_{\pm,{\epsilon}/{2}}(\omega_0)$ respectively. Let us define the partial Dirichlet to Neumann map denoted by $\Ld_{A,q}: H_0^{\f{3}{2},\f{3}{4}}(\Sigma)\to H^{\f{1}{2},\f{1}{4}}(\Sigma_{-,{\epsilon}/{2}}(\omega_0))$ as
		\begin{align}\label{partial DN map}
		\Lambda_{A,q}(f)=\lb \partial_{\nu}u+2(A\cdot\nu)u\rb \big\vert_{\Sigma_{-,{\epsilon}/{2}}(\omega_0)}   
		\end{align}
		 We now state the main result of this article.
		\begin{theorem}\label{Main Theorem}
			Let $(A_i,q_i)\in\mc{M}(m),\ i = 1, 2$	and $T>\tn{diam}\ \Omega$, where $\Omega\subset\mb{R}^n$ is a bounded $C^2$ smooth domain for $n\ge 2$. We denote by $\Lambda_i$ the partial DN map corresponding to $\mc{L}_{A_i,q_i}$ as defined in \eqref{partial DN map}. Under the assumption $A_1\vert_{\Sigma}=A_2\vert_{\Sigma}$ and $\nabla_x\cdot A_1=\nabla_x\cdot A_2$ in $Q$, we have the following estimates for some positive constants $C,\alpha_1,\alpha_2,\beta_1$ and $\beta_2$ depending on $m$ and $Q$ 
			\begin{align*}
			&\|A_1-A_2\|_{L^2(Q)}\le C\left(\|\Ld_1-\Ld_2\|^{\alpha_1}+\left|\log|\log\|\Lambda_1-\Lambda_2\||\right|^{\alpha_2}\right),\\
			&\|q_1-q_2\|_{L^2(Q)}\le C\left(\|\Ld_1-\Ld_2\|^{\beta_1}+\left|\log\left|\log|\log\|\Lambda_1-\Lambda_2\|\right||\right|^{\beta_2}\right).
			\end{align*}	
		\end{theorem}
		We remark here that, in the recent work \cite{Bellassoued_Ben Fraj}, Bellassoued and Ben Fraj discussed the stability aspects of determining the time-dependent coefficients appearing in the convection-diffusion equation and proved logarithmic and double logarithmic stability results for the convection term and density coefficient respectively. Moreover the Neumann measurements there are taken on any arbitrary part of the lateral boundary $\Sigma$; but the coefficients in \cite{Bellassoued_Ben Fraj} are assumed to be known in an open set containing $\Sigma$ which is essential to apply a local unique continuation result near the boundary. In contrast to \cite{Bellassoued_Ben Fraj}, we consider coefficients which agree only on the lateral boundary. Although we work with the Neumann data measured on a particular subset of $\Sigma$, which is slightly more than half of the boundary and obtain double and triple logarithmic stability estimates for $A$ and $q$ respectively.

		\section{Boundary and interior Carleman estimates}\label{Boundary Carleman estimate}
		We prove here a boundary Carleman estimate involving for the operator $ \la_{A,q} $ 
		which will be used to control the boundary terms appearing in the integral identity given by \eqref{integral identity} where no information is given. Now, we choose $x_0\in\Rb^n$ such that $\inf\limits_{x\in\overline{\Omega}}\left(x+x_0\right)\cdot\og>0$. For this choice of $\og$ and $x_0$, the derivation of Carleman estimate goes as follows.
		
		\begin{theorem}\label{Boundary Carleman estimate Theorem}
			Let $\phi(t,x)=\ld^2 t+\ld x\cdot \omega$ and $u\in C^2(\overline{Q})$ with $u(0,\cdot)=0$ and $u|_{\Sigma}=0$. For $(A,q)\in\mc{M}(m)$ there exist $\ld_1,C>0,$ depending only on  $m$ and $Q$ such that  
			\begin{align}\label{Carleman}
			\begin{aligned}
			&\int_{Q} e^{-2\phi(t,x)} \lb \lambda^{2}	\lvert u(t,x)\rvert^{2}+\lvert \nabla_{x} u(t,x)\rvert^{2}\rb\ \d x \d t
			+\int_{\Omega}e^{-2\phi(T,x)}\lb \lambda \lvert u(T,x)\rvert^{2}+\lvert \nabla_{x}u(T,x)\rvert^{2}\rb\ \d x\\
			&\qd +\lambda  \int_{\Sigma_{+,\og}}e^{-2\phi(t,x)}\og\cdot\nu(x)\vert\dd_{\nu}u(t,x)\vert^2\ \d S_{x}\d t\leq C \int_{Q}e^{-2\phi(t,x)}\lvert \mc{L}_{A,q}u(t,x)\rvert^{2}\ \d x\d t\\
			&\qd \qd\qd\qd\qd\qd\qd\qd\qd\qd\qd\qd\qd+ C\lambda\int_{\Sigma_{-,\og}}e^{-2\phi(t,x)}\lvert\og\cdot\nu(x)\rvert\lvert\dd_{\nu}u(t,x)\rvert^2\ \d S_{x}\d t
			\end{aligned}
			\end{align}
			holds for all $\ld\ge\ld_1$. 
			\begin{proof}
			We have to convexify the Carleman weight $\phi$ appropriately due to the presence of first order derivatives in $\mc{L}_{A,q}$. For a proof of the boundary Carleman estimate, we refer to \cite{Caro_Kian_Convection_nonlinear} where the following convexified weight has been considered for $s>0$
						\begin{align*}
				\phi_{s}(t,x):=\ld^2t+\ld x\cdot\og-\f{s\left((x+x_0)\cdot\og\right)^2}{2}
				\end{align*}  
				where $x_{0}\in \mathbb{R}^{n}$ is as mentioned in the line just before the statement of Theorm \ref{Boundary Carleman estimate Theorem}. 
			\end{proof}  
		\end{theorem}
		
		We write down the interior Carleman estimate which easily  follows from Theorem \ref{Boundary Carleman estimate Theorem} and  will be used to construct the geometric optics solutions.
		\begin{corollary}\tn{(Interior Carleman estimate).}\label{Interior Carleman estimate}
			For $(A,q)\in\mc{M}(m)$ there exist $\ld_1,C>0$ depending only on  $m$ and $Q$ such that the following estimate holds for $u\in C_c^{\iy}(Q)$ and $\ld\ge\ld_1$
			\begin{align*}
			\begin{aligned}
			\int_{Q} e^{-2\phi(t,x)} \lb \lambda^{2}	\lvert u(t,x)\rvert^{2}+\lvert \nabla_{x} u(t,x)\rvert^{2}\rb\ \d x\d t
			\leq C \int_{Q}e^{-2\phi(t,x)}\lvert \mc{L}_{A,q}u(t,x)\rvert^{2}\ \d x\d t.
			\end{aligned}
			\end{align*}
		\end{corollary}

		\section{Construction of geometric optics solutions}\label{Contruction of go solutions}
		In this section, we construct the geometric optics solutions for the parabolic operator $\mc{L}_{A,q}$ and its formal $L^{2}$ adjoint $\mc{L}^{*}_{A,q}=\mc{L}_{-A,\overline{q}}$. For $\ld>0$, let  $\phi(t,x):=\ld^{2}t+\ld x\cdot\omega$ be the weight function. Then we construct the geometric optics solutions $u$ and $v$ for the operators $\mc{L}_{A,q}$ and $\mc{L}^{*}_{A,q}$ respectively which have the following forms
		\begin{align}\label{Construction for GO} 
		\begin{aligned}
		&	u(t,x)=e^{\phi(t,x)}\left(B_{g}+R_{g}\right)(t,x),\ \ \ \\
		\tn{and}\ \ &	v(t,x)=e^{-\phi(t,x)}\left(B_{d}+R_{d}\right)(t,x).
		\end{aligned}
		\end{align} 
		Next we show that for $\ld$ large enough, the remainder terms $R_{g}$ and $R_{d}$ can be estimated in terms of their principal terms $B_{g}$ and $B_{d}$ respectively. The decay of $R_{d}$ and $R_{g}$ in  $\ld$ will be crucial to derive stability results from an integral identity obtained by using the solution to an adjoint problem and the given data. 
		
		\vspace*{3mm}
		We start with some definition and notations. For $m\in\mb{R}$, we define $L^2(0,T;H^m_\ld(\mb R^n))$ by
		\[L^2(0,T;H^m_\ld(\mb R^n)):=\{u(t,\cdot)\in\mc{S}'(\mb R^n);\left(\ld^2+|\xi|^2\right)^{m/2}\mc{F}_xu(t,\xi)\in L^2((0,T)\times\mb R^n)\},\]
		equipped with the norm
		\[\|u\|^2_{L^2(0,T;H^m_\ld(\mb R^n))}=\int_0^T\int_{\mb R^n}\left(\ld^2+|\xi|^2\right)^{m}|\mc{F}_xu(t,\xi)|^2\ \D\xi dt, \] 
		where $\mc{S}'(\mb R^n)$ denote the space of tempered distributions on $\mb R^n$ and $\mc{F}_x$ is the Fourier transform with respect to the space variables. For $0<\delta<<1$, we consider a sequence $\eta_\delta\in C_{c}^{\infty}(\delta, T-\delta)$ such that   
		\[\mbox{$\eta_\delta\equiv 1$ on $[2\delta,T-2\delta]$ and, 	 }\ \|\eta_\delta\|_{W^{k,\iy}(\mb R)}\le C\delta^{-k},\qd\tn{for }k\in\mb{N}.\] 
		\begin{theorem}\label{GO solutions Theorem}
			Let   $\delta\in(0,T/4)$, $\mathcal{L}_{A,q}$ be as in \eqref{Definition of LAq} and for $\omega\in \Sb^{n-1}$, let $\phi(t,x)=\lambda^{2}t+\lambda x\cdot\omega$.  
			\begin{enumerate}
				\item (Exponentially growing solutions) \label{Growing GO Theorem}
				For   $(\tau,\xi)\in\mb R^{1+n}$ such that $\og\cdot\xi=0$, $( A,q)\in\mc{M}(m)$ and $D\in W^{2,\iy}(Q)^{n}$ with $\|D\|_{W^{2,\iy}(Q)}\le C_0$, there exists $\ld_0>0$ depending on $m,C_0$ and $Q$ such that for $ \ld\ge \ld_0$, we can find  $v_{g}\in H^{2,1}\left((0,T)\times\Omega\right)$ solution to  	\begin{align*}
				\begin{aligned}
				\begin{cases}
				\mc{L}_{A,q}v(t,x)=0,\qd (t,x)\in Q,\\
				v(0,x)=0,\qd x\in \Omega
				\end{cases}
				\end{aligned}
				\end{align*} 
				taking the form 
				\begin{align}\label{Growing Go solutions}
				v_{g}(t,x)=e^{\phi(t,x)}\left(B_{g}(t,x)+R_{g}(t,x)\right)
				\end{align}
				where $B_{g}(t,x)$ is given by 
				\begin{align}\label{Expression for Bg}
				B_{g}(t,x)=\eta_\delta(t)\f{\xi}{|\xi|}\cdot\nabla_x\left(e^{-i(\tau,\xi)\cdot (t,x)}e^{\left(\int_\mb R\og\cdot D(t,x+s\og)\ \d s\right)}\right)e^{\left(\int_0^\iy\omega\cdot A(t,x+s\omega)\ \d s\right)}
				\end{align}
				and $R_{g}$ satisfies the following estimate 
				\begin{align}\label{Estimate for Rg}
				\|R_g\|_{L^2(0,T;H^{k}(\Omega))}\le C\ld^{-1+k}\delta^{-3}\langle\tau,\xi\rangle^3 ,\ \ \mbox{for $k\in\{0,1,2\}$}.
				\end{align}
				\item (Exponentially decaying solutions)\label{Decaying GO Theorem} For $(\mc A,q)\in\mc{M}(m)$, there exists $\ld_0>0$ depending on $m$ and $Q$ such that for $ \ld\ge \ld_0$, we can find $v_{d}\in H^{2,1}\left((0,T)\times\Omega\right)$ solution to 
				\begin{align*}
				\begin{aligned}
				\begin{cases}
				\mc{L}^{*}_{A,q}v(t,x)=0,\qd (t,x)\in Q,\\
				v(T,x)=0,\qd x\in \Omega.
				\end{cases}
				\end{aligned}
				\end{align*}
				taking the form
				\begin{align}\label{Decaying GO solutions}
				v_{d}(t,x)=e^{-\phi(t,x)}\left(B_d(t,x)+R_d(t,x)\right)
				\end{align}
				where $B_{d}(t,x)$ is given by 
				\begin{align}\label{Expression for Bd}
				B_{d}(t,x)=\eta_\delta(t)e^{\left(\int_0^\iy\omega\cdot A(t,x+s\omega)\ \d s\right)}
				\end{align}
				and  $R_{d}$ satisfies the following estimates 	
				\begin{align*}
				\|R_d\|_{L^2(0,T;H^{k}(\Omega))}\le C\ld^{-1+k}\delta^{-3},\ \ \mbox{for $k\in\{0,1,2\}$}. 
				\end{align*}
			\end{enumerate}
		\end{theorem}
		The proof of the above Theorem relies mainly on some arguments from functional analysis as we need to consider appropriate functional which would be extended and identified by Hahn-Banach and Riesz representation theorems. But continuity of such functional would be possible once we obtain suitable negative order Carleman estimates. Thus we state the following Carleman estimate in a negative order Sobolev space and proof of this estimate follows from very standard arguments (see \cite{Bellassoued_Ben Fraj,Bellassoued_Jellali_Yamamoto_Lipschitz_stability_hyperbolic,Bellassoued_Jellali_Yamamoto_stability_hyperbolic,Caro_Kian_Convection_nonlinear,Ferriera_Kenig_Sjostrand_Uhlmann_magnetic} and references therein). To state the Carleman estimate, we define the conjugated operator $\mc{P}_{\ld}$ by
		\[\mc{P}_{\ld}:= e^{-\phi}\mc{L}_{A,q}e^{\phi}.\]
		
		\begin{proposition}{\tn{(Shifting the index).}}
			Let $A,q,\phi$ and $\mathcal{L}_{A,q}$ be as in Theorem \ref{GO solutions Theorem}. \begin{enumerate}
				\item Let $\mc{P}_{\ld}:= e^{-\phi}\mc{L}_{A,q}e^{\phi},$ then there exists $\ld_0$ and $C>0$ such that for $ u\in C^1([0,T];C_c^{\iy}(\Og))$ with $u(T,\cdot)=0$ and $\ld\ge\ld_0$, we have 
				\begin{align}\label{shifting}
				\|u\|_{L^2(0,T;H^{-1}_{\ld}(\mb R^n))}\le C \|\mc{P}_{\ld}u\|_{L^2(0,T;H^{-2}_{\ld}(\mb R^n))}.
				\end{align}
				\item Let $\mc{P}^*_{\ld}:= e^{\phi}\mc{L}^*_{A,q}e^{-\phi},$ then there exists $\ld_0$ and $C>0$ such that for $ u\in C^1([0,T];C_c^{\iy}(\Og))$ with $u(0,\cdot)=0$ and $\ld\ge\ld_0$, we have 
				\begin{align}\label{shifting in conjugate case}
				\|u\|_{L^2(0,T;H^{-1}_{\ld}(\mb R^n))}\le C \|\mc{P}^*_{\ld}u\|_{L^2(0,T;H^{-2}_{\ld}(\mb R^n))}.
				\end{align}
			\end{enumerate}  
			
		\end{proposition}	
		
		For the sake of completeness, we prove the following standard proposition which will lead to the proof of Theorem \ref{GO solutions Theorem}. 
		\begin{proposition}
			For $f\in L^2\left(0,T;H^1(\mb{R}^n)\right)$ there exists $w\in H^{2,1}\left((0,T)\times\mb{R}^n\right)$ solving the IVP
			\begin{align*}
			\begin{aligned}
			\begin{cases}
			\mc{P}_{\ld}w=f,\qd \tn {in }Q,\\
			w(0,x)=0,\   x\in\Omega
			\end{cases}
			\end{aligned}
			\end{align*}
			and it satisfies,  $\|w\|_{L^2(0,T;H^2_\ld(\Omega))}\le C\|f\|_{L^2(0,T;H^1_\ld(\Omega))}$ for some $C>0$ depending only on $m$ and $Q$.
		\end{proposition}
		
		\begin{proof}
			Consider the space $\mc{W}:=\{\mc{P}_{\ld}^{*} u : u\in C^1\left([0,T];C^\iy_c(\Omega)\right) \mbox{and}\ u(T,\cdot)=0\}$ equipped with the norm $\|\cdot\|_{L^2(0,T;H^{-2}_{\ld}(\mb R^n))}$. Now for $f\in L^2\left(0,T;H^1(\mb{R}^n)\right)$,  define a functional $T_{f}$  on $\mc{W}$ by 
			\begin{align}
			T_{f}(\mc{P}_{\ld}^* u):=\int_{\mb R^{1+n}}f(t,x)u(t,x) \ \d x\d t.  \label{defintion}   
			\end{align}
			Now using  \eqref{shifting}, we have that 
			$T_f$ is a continuous linear functional on  $\mc{W}$ with 
			\begin{align}
			\|T_{f}\|_{\mc{W}}\le C\|f\|_{L^2(0,T;H^{1}_{\ld}(\mb R^n))}\label{continuity on a subspace}
			\end{align}    
			and hence by the Hahn-Banach Theorem, $T_{f}$ can be extended to ${L^2(0,T;H^{-2}_{\ld}(\mb R^n))}$  which will be still denoted by $T_f$ and satisfy \eqref{continuity on a subspace}. Finally using the Riesz representation theorem, there exists $w\in{L^2(0,T;H^{2}_{\ld}(\mb R^n))}$ such that for $ v\in{L^2(0,T;H^{-2}_{\ld}(\mb R^n))}$ we have
			\begin{align}
			T_f(v)=\int_{\mb R^{1+n}}v(t,x)w(t,x)\ \d x \d t.\label{Reisz}
			\end{align}
			Now combining \eqref{defintion} and \eqref{Reisz} we obtain 
			\begin{align}\label{conseq of Reisz}
			\int_{\mb R^{1+n}}\mc{P}_{\ld}^*v(t,x)w(t,x) \ \d x \d t=\int_{\mb R^{1+n}}v(t,x) f(t,x) \ \d x \d t,
			\end{align}
			for $ v\in C^1\left([0,T];C^\iy_c(\Omega)\right)$ such that $v(T,\cdot)=0$ in $\Omega$. This gives us 
			$\mc{P}_\ld w=f$ in $Q$. Since  $f\in L^2\left(0,T;H^1(\mb{R}^n)\right)$ we have $w\in H^1\left(0,T;L^2(\mb{R}^n)\right)$. 
			Now for $v\in C^1\left([0,T];C^\iy_c(\Omega)\right)$ satisfying $v(T,\cdot)=0$, we use the  integration by parts to \eqref{conseq of Reisz} to obtain
			\[\int_{\Omega}w(0,x)v(0,x)\ \d x=0.\]  
			Hence $w(0,\cdot)=0$ in $\Omega$.  
			Finally we use \eqref{continuity on a subspace} and \eqref{Reisz} to get
			\[\|T_f\|_{L^2(0,T;H^{-2}_\ld(\Omega))}=\|w\|_{L^2(0,T;H^2_\ld(\Omega))}\le C\|f\|_{L^2(0,T;H^1_\ld(\Omega))}.\]
		\end{proof}
		\subsection{Proof of Theorem \ref{GO solutions Theorem}}
		
		First observe that 	
		\begin{align}
		\mc{L}_{A,q}\left(e^{\phi}v\right)=e^{\phi}\left(\mc{L}_{A,q}v-2\ld\omega\cdot(\nabla_x+A)v\right).\nn
		\end{align}
		Now using the expressions for $v_{g}$, $B_{g}$ and $\mc{L}_{A,q}v_{g}=0$, we have that 
		the remainder terms $R_g$ solves
		\begin{align}\label{remainder term}
		\mc{P}_{\ld}R_g=-\mc{L}_{A,q}B_g
		\end{align}
		where $\mc{P}_{\lambda}:=e^{-\phi}\mc{L}_{A,q}e^{\phi}$ is the conjugated operator as defined earlier and $B_{g}$ solves the following transport equation \[\omega\cdot\lb \nabla_{x}+ A\rb B_{g}=0.\] 
		Hence from \eqref{remainder term} and Proposition 4.3, it is clear that the remainder term $R_{g}$  satisfies 
		\begin{align*}
		\|R_g\|_{L^2(0,T;H^k(\Omega))}\le C\ld^{-1+k}\|B_g\|_{H^3(Q)}, \  \tn{for }k\in\{0,1,2\}. 
		\end{align*}	
		This completes the proof for the construction of exponentially growing solutions to $\mc{L}_{A,q}v=0$. One can carry out exactly same set of arguments to prove the existence of exponentially decaying solutions having the form given by equation \eqref{Decaying GO solutions} and  solution to  $\mc{L}_{A,q}^{*}v=0$. This complete the proof of Theorem \ref{GO solutions Theorem}. 
		Hence from \eqref{remainder term} and Proposition 4.3, it is clear that the remainder term $R_{g}$  satisfies 
		\begin{align*}
		\|R_g\|_{L^2(0,T;H^k(\Omega))}\le C\ld^{-1+k}\|B_g\|_{H^3(Q)}, \  \tn{for }k\in\{0,1,2\}. 
		\end{align*}	
		This completes the proof for the construction of exponentially growing solutions to $\mc{L}_{A,q}v=0$. One can carry out exactly same set of arguments to prove the existence of exponentially decaying solutions having the form given by equation \eqref{Decaying GO solutions} and  solution to  $\mc{L}_{A,q}^{*}v=0$. This complete the proof of Theorem \ref{GO solutions Theorem}.

		\section{Proof of theorem \ref{Main Theorem}}
		
		In this section, we prove the main result on stability for the first and zeroth order coefficients. But first we derive an integral identity using Green's formula where we will plug in the geometric optics solutions constructed before. We simultaneously consider exponentially growing and decaying solutions to avoid any exponential term or boundary terms at initial or final time. To be precise, we construct $u_2$ and $v$ as the exponentially growing and decaying solutions for the operators $\mc{L}_{A_2,q_2}$ and $\mc{L}_{-A_1,\overline q_1}$ respectively by using Theorem \ref{GO solutions Theorem}. Taking $0<\delta<<1$, $(\tau,\xi)\in\mb{R}^{1+n}$ with $\xi\cdot\omega=0$ and $D(t,x)=A(t,x):=\left(A_1-A_2\right)(t,x)$ we have
		\begin{align}
		&u_{2}(t,x)=e^{\phi(t,x)}\left(B_{2}+R_{2}\right)(t,x),\label{GO 1}\\
		&\tn{and}\ \ v(t,x)=e^{-\phi(t,x)}\left(B+R\right)(t,x)\label{GO 2}
		\end{align}
		where
		
		\vspace*{-0.8cm}
		
		\begin{align*}
		&B_{2}(t,x)=\eta_\delta (t)\f{\xi}{|\xi|}\cdot\nabla_{x}\left(e^{-i\lb t\tau+x\cdot\xi\rb}e^{\left(\int_{\mb R}\omega\cdot A(t,x+s\omega)\ \d s\right)}\right)e^{\left(\int_{0}^{\iy}\omega\cdot A_2(t,x+s\omega)\ \d s\right)},\nn\\
		\qd&B(t,x)=\eta_\delta (t)e^{\left(-\int_{0}^{\iy}\omega\cdot A_1(t,x+s\omega)\ \d s\right)}
		\end{align*}
		and $R_{2},R\in L^2(0,T;H^2_\ld(\Omega))$ satisfy
		\begin{align}\label{remainder terms}
		\|R_{2}\|_{L^2(0,T;H^k_\ld(\Omega))}\le C\ld^{-1+k}\delta^{-3}\langle\tau,\xi\rangle^3\tn{ and}\  \|R\|_{L^2(0,T;H^k_\ld(\Omega))}\le C\ld^{-1+k}\delta^{-3},\ \tn{for }k\in\{0,1,2\}.
		\end{align}
		Hence, we have for some $\beta>0$
		\begin{align}\label{estimate for geometric optics}
		\|u_2\|_{L^2(\Sigma)}\le e^{\beta\ld}\delta^{-3}\langle\tau,\xi\rangle^3. 
		\end{align}
		Also, observe that  $u_2(0,x)=v(T,x)=0$ for $x\in\Omega$. 
		Now, consider $u_1$ to be the solution of the IBVP 
		\begin{align*}
		\begin{cases}
		\mc{L}_{A_1, q_1}w(t,x)=0,\qd (t,x)\in Q,\\
		w(0,x)=0,\ \ \  x\in\Omega\\
		\ w(t,x)=u_2(t,x),\ \ \ (t,x)\in \Sigma
		\end{cases}  		
		\end{align*}
		and define $u:=u_1-u_2$ in $Q$. Then we get
		\begin{align}\label{PDE for the diff of soln 1}
		\begin{aligned}
		\begin{cases}
		\mc{L}_{A_1,q_1}u(t,x)=2A(t,x)\cdot\nabla_xu_2(t,x)+\widetilde q(t,x)u_2(t,x),\ \ \ (t,x)\in Q,\\
		u(0,x)=0,\ \ \ x\in\Omega\\
		u(t,x)=0, \ \ \ (t,x)\in \Sigma
		\end{cases}
		\end{aligned}
		\end{align}
		where 
		\begin{align*}
		\begin{aligned}
		&A(t,x)\equiv \{A^j(t,x)\}_{1\le j\le n}:=A_1(t,x)-A_2(t,x),\  \ \ \widetilde{q}(t,x):=\widetilde{q}_1(t,x)-\widetilde{q}_2(t,x) \ \ \mbox{and}\\
		&q(t,x):=q_1(t,x)-q_2(t,x).
		\end{aligned}
		\end{align*}
		Using Green's formula, we have 
		\begin{align*}
		\begin{aligned}
	&	\int_{Q}(\mc{L}_{A_1,q_1}u)(t,x)\overline v(t,x) \ \d x \d t-\int_{Q}u(t,x){\mc{L}_{-A_1,\overline q_1}\overline v(t,x)}\ \d x \d t=\int_\Omega u(T,x)\overline{v}(T,x)\ \d x \\
		&\ \ \ \ \ \ \ -\int_\Omega u(0,x)\overline{v}(0,x)\ \d x  -\int_{\Sigma}\overline v(t,x)\dd_{\nu}u(t,x)\ \d S_{x} \d t+\int_{\Sigma}u(t,x)\overline{\dd_{\nu}v(t,x)}\ \d S_{x} \d t,
		\end{aligned}
		\end{align*}
		which after using \eqref{PDE for the diff of soln 1} becomes
		\begin{align}\label{integral identity}
		2\int_{Q}(A(t,x)\cdot\nabla_xu_2(t,x))\overline{v}(t,x)\ \d x\d t+\int_{Q}\widetilde{q}(t,x)u_2(t,x)\overline{v}(t,x)\ \d x\d t=-\int_{\Sigma}\overline v(t,x)\dd_{\nu}u(t,x)\ \d S_{x} \d t.
		\end{align}
		We observe from \eqref{GO 1}
		\begin{align}\label{lhs in integral inequality 1.a}
		A(t,x)\cdot\nabla_xu_2(t,x)=e^{\phi}\left(\ld \omega\cdot A(t,x)B_{2}(t,x)+m(t,x)\right),
		\end{align}
		for some $ m\in L^2(0,T;H^1(\Og))$ satisfying the following estimate 
		\begin{align}\label{lhs in integral inequality 1.b}
		\|m\|_{L^2(0,T;H^k(\Og))}\le C\ld^k\delta^{-3}\langle\tau,\xi\rangle^3, \tn{ for }k\in\{0,1\}\ \mbox{and}\ \ld\ge\ld_0.
		\end{align} 		
		An application of the Cauchy-Schwartz inequality together with \eqref{GO 2},\eqref{remainder terms},\eqref{lhs in integral inequality 1.a} and \eqref{lhs in integral inequality 1.b}  give us 
		\begin{align}
		\left(A\cdot\nabla_xu_2\right)(t,x)\overline v(t,x)=\left(\ld (\omega\cdot A)B_{2}B+n\right)(t,x)\label{multiplication of principal terms}
		\end{align}
		for some $n\in L^1(Q)$ satisfying
		\begin{align}
		\|n\|_{L^1(Q)}\le C\delta^{-6}\langle\tau,\xi\rangle^3.
		\end{align}
		Also from \eqref{GO 1} and \eqref{GO 2}, it is clear that for $\ld\ge\ld_0$
		\begin{align}
		\left\vert\int_{Q}\widetilde{q}(t,x)u_2(t,x)\overline{v}(t,x)\ \d x\d t\right\vert\le C\delta^{-6}\langle\tau,\xi\rangle^3.\label{lhs in integral inequality 2}
		\end{align}
		Now, we find an upper bound for the right hand side of \eqref{integral identity} using the boundary Carleman estimate \eqref{Carleman}. We observe  
		\begin{align*}
		\int_{\Sigma}\overline v(t,x)\dd_{\nu}u(t,x)\ \d S_{x}\d t=\int_{\Sigma_{-,{\epsilon}/{2}}(\omega_0)}\overline v(t,x)\dd_{\nu}u(t,x)\ \d S_{x}\d t+\int_{\Sigma_{+,{\epsilon}/{2}}(\omega_0)}\overline v(t,x)\dd_{\nu}u(t,x)\ \d S_{x}\d t
		\end{align*}
		where $\Sigma_{+,{\epsilon}/{2}}(\omega_0)$ is the part of lateral boundary where we do not have any knowledge of Neumann measurements. Although the contribution from that part can be estimated by using the boundary Carleman estimate. Meanwhile for $\omega\in\mb{S}^{n-1}$ satisfying $|\omega-\omega_0|\le\f{\ep}{2}$, we have $\Sigma_{+,{\epsilon}/{2}}(\omega_0)\subseteq \Sigma_{+}(\omega)$ and $\Sigma_{-}(\omega)\subseteq \Sigma_{-,{\epsilon}/{2}}(\omega_0)$.
		Thus, we get from \eqref{estimate for geometric optics}  	 
		\begin{align}
		\|\overline v\dd_{\nu}u\|_{L^1\left(\Sigma_{-,{\epsilon}/{2}}(\omega_0)\right)}\le Ce^{\beta\ld}\delta^{-3}\left\|\dd_{\nu}u\right\|_{{L^2}\left(\Sigma_{-,{\epsilon}/{2}}(\omega_0)\right)}\le Ce^{\beta\ld}\delta^{-6}\langle \tau,\xi\rangle^3\|\Ld_1-\Ld_2\|\label{rhs of integral identity 1}
		\end{align}
		and
		\begin{align*}
		\begin{aligned}
		\|\overline v\dd_{\nu}u\|_{L^1\left(\Sigma_{+,{\epsilon}/{2}}(\omega_0)\right)}&\le C\delta^{-3}\|e^{-\phi}\dd_{\nu}u\|_{L^2\left(\Sigma_{+,{\epsilon}/{2}}(\omega_0)\right)}\le\f{2C\delta^{-3}}{\rt\ep}\rt{\int_{\Sigma_{+,{\epsilon}/{2}}(\omega_0)}e^{-2\phi}|\og\cdot\nu(x)||\del_{\nu}u|^2\ \d S_x}\nn\\
		&\le\f{2C\delta^{-3}}{\rt\ep}\rt{\int_{\Sigma_{+}(\omega)}e^{-2\phi}|\og\cdot\nu(x)||\del_{\nu}u|^2\ \d S_x} \nn
		\\&  \le \f{2C\delta^{-3}}{\rt{\ep\ld}} \left(\|e^{-\phi}\mc{L}_{A_1,q_1}(u)\|_{L^2(Q)}+\rt\ld\|e^{-\phi}\del_{\nu}u\|_{L^2\left(\Sigma_{-}(\og)\right)}\right). 
		\end{aligned}
		\end{align*}
		Using \eqref{Carleman}, we have 
		\begin{align*}
		\|\overline v\dd_{\nu}u\|_{L^1\left(\Sigma_{+,{\epsilon}/{2}}(\omega_0)\right)} \le \f{2C\delta^{-3}}{\rt{\ep\ld}} \left(\|e^{-\phi}\left(2A\cdot\nabla_xu_2+\widetilde qu_2\right)\|_{L^2(Q)}+\rt\ld\|e^{-\phi}\del_{\nu}u\|_{L^2\left(\Sigma_{-,{\epsilon}/{2}}(\omega_0)\right)}\right). 
		\end{align*}
		Finally	using \eqref{remainder terms},\eqref{estimate for geometric optics},\eqref{lhs in integral inequality 1.a} and \eqref{lhs in integral inequality 1.b}, we get 
		\begin{align}\label{unknown part}
		\|\overline v\dd_{\nu}u\|_{L^1\left(\Sigma_{+,{\epsilon}/{2}}(\omega_0)\right)}	\le C\delta^{-6}\langle\tau,\xi\rangle^3\left(\rt\ld+e^{\beta\ld}\|\Lambda_1-\Lambda_2\|\right).
		\end{align}
		After dividing the integral identity \eqref{integral identity} by large $\ld$ and using Equations \eqref{lhs in integral inequality 1.a} to \eqref{unknown part}, we obtain
		\begin{align}\label{main estimate 1}
		\left\vert\int_{Q}(\omega\cdot A)(t,x)B_{2}(t,x)B(t,x)\ \d x\d t\right\vert\le C\left(\f{1}{\rt\ld}+e^{\beta\ld}\|\Ld_1-\Ld_2\|\right)\delta^{-6}\langle\tau,\xi\rangle^3.
		\end{align}
		Next, we relate the integral in the left hand side of \eqref{main estimate 1} with the Fourier-transform of $A$ as is done in \cite{Kian_Soccorsi_Holder_stability_Scrodinger}. 
		\begin{align*}
		&\int_{Q}(\omega\cdot A)(t,x)B_{2}(t,x)B(t,x)\ \d x\d t\\
		&\ \ \ 
		=\int_{\mb R^{1+n}}(\omega\cdot A)(t,x)\eta^2_\delta (t)e^{\lb -\int_{0}^{\iy}\omega\cdot A(t,x+s\omega)\ \d s\rb}\f{\xi}{|\xi|}\cdot\nabla_{x}\left(e^{\int_{\mb R}\omega\cdot A(t,x+s\omega)\ \d s} e^{-i\lb t\tau+x\cdot\xi\rb}\right) \ \d x\d t .
		\end{align*}
		Now, using the	decomposition  $\mathbb{R}^{n}=\mathbb{R}\omega\oplus \omega^{\perp}$, we write $x:=x_{\perp}+s\omega$ such that  $x_{\perp}\in \omega^{\perp}$ and denote  $$f(s,t,x_{\perp}):=\int_{s}^{\iy}\omega\cdot A(t,x_{\perp}+\mu\omega)\ \d\mu .$$ Using these in the previous equation, we have
		\begin{align*}
		\begin{aligned}
		&\int_{Q}(\omega\cdot A)(t,x)B_{2}(t,x)B(t,x)\ \d x\d t\\
		&\ \ \ =-\int_{\mb R}\eta^2_\delta(t)\int_{\og^{\perp}}\f{\xi}{|\xi|}\cdot\nabla_{x}\left(e^{\int_{\mb R}\omega\cdot A(t,x_{\perp}+s\omega)\ \d s}e^{-i\lb t\tau+x_{\perp}\cdot\xi\rb}\right)\lb \int_{\mb R}f'(s,t,x_{\perp})e^{-f(s,t,x_{\perp})}\ \d s\rb  \d x_{\perp}\d t\nn\\
		\end{aligned}
		\end{align*}
		where  $\d x_{\perp}$ stands for the surface measure on $\omega^{\perp}$. Now using the Fundamental theorem of calculus and integration by parts, we get 
		\begin{align*}
		&\int_{Q}(\omega\cdot A)(t,x)B_{2}(t,x)B(t,x)\ \d x\d t\\	&=\int_{\mb R}\eta^2_\delta (t)\int_{\og^{\perp}}\f{\xi}{|\xi|}\cdot\nabla_{x}\left(e^{\int_{\mb R}\omega\cdot A(t,x_{\perp}+s\omega)\ ds}e^{-i\lb t\tau+x_{\perp}\cdot\xi\rb }\right)\left(e^{\left(-\int_{\mb R}\omega\cdot A(t,x_{\perp}+\mu\omega)\ d\mu\right)}-1\right)\d x_{\perp}\d t,\nn\\
		&=\int_{\mb R}\eta^2_\delta (t)\int_{\og^{\perp}}e^{-i\lb t\tau+x_{\perp}\cdot\xi\rb}\f{\xi}{|\xi|}\cdot\nabla_{x}\left(\int_{\mb R}\omega\cdot A(t,x_{\perp}+s\omega)\ \d s\right)\d x_{\perp}\d t\\
		&=\int_{\mb R}\eta^2_\delta (t)\int_{\mb R^n}e^{-i\lb t\tau+x\cdot\xi\rb}\f{\xi}{|\xi|}\cdot\nabla_{x} (\og\cdot A)(t,x)\ \d x \d t=i|\xi|\int_{\mb R}\eta^2_\delta (t)\int_{\mb R^n}e^{-i\lb t\tau+x\cdot\xi\rb}\omega\cdot A(t,x)\ \d x\d t.
		\end{align*}
		Finally, we obtain
		\begin{align}\label{FT of A}
		\int_{Q}(\omega\cdot A)(t,x)B_{2}(t,x)B(t,x)\ \d x\d t=i|\xi|\ \widehat{\eta^2_\delta\omega\cdot A} (\tau,\xi),  
		\end{align}
		where $\omega\cdot\xi=0$.
		
		Let us consider the spherical cap $\mc{C}_{\og_0}:=\{\og\in\mb{S}^{n-1};|\og-\og_0|<\f{\ep}{2}\}$ and the set $\mc{H}=\cup_{\og\in\mc{C}_{\og_0}}\mc{H}_{\og}$ where $\mc{H}_\og$ is the plane passing through origin and perpendicular to $\og$. Now for $(\tau,\xi)\in\mb{R}\times\mc{H}$,  $\ld\ge\ld_0$ and choosing $\og(\xi)\in\mc{C}_{\og_0}$ such that $\og(\xi)\cdot\xi=0$ in  \eqref{main estimate 1} and \eqref{FT of A}, we get 
		\begin{align}\label{main estimate 3}
		\left\vert\lb\eta_\delta^2\dd_k\og(\xi)\cdot A\rb\vspace*{-.5mm}\widehat\ \  (\tau,\xi)\right\vert\le C\left(\f{1}{\rt\ld}+e^{\beta\ld}\|\Ld_1-\Ld_2\|\right)\delta^{-6}\langle\tau,\xi\rangle^3 
		\end{align}
		where $\dd_{k}$ denote the partial derivative with respect to the space variable $x_k$ for $k\in\{1,2,\cdots,n\}$.
		With the help of \eqref{main estimate 3}, we aim to establish the desired stability estimate via the Fourier inversion. Because of the  Vessella's conditional stability result \cite{Vessella}, it is enough to  derive  a uniform estimate for the Fourier transform of $\eta_\delta^2 A$  over an open cone only. This we do in the following lemma.     We have crucially used the divergence free assumption on $A$ to prove the following lemma.
		\begin{lemma}\label{Inversion} 
			If $div_{x}(A)=0$ then for $(\tau,\xi)\in\mb R\times\mc{C}$ and $k\in\{1,2,...,n\}$ we have 
			\begin{align}
			\vert\widehat{\eta_\delta^2\dd_kA}(\tau,\xi)\vert\le C\left(\f{1}{\rt\ld}+e^{\beta\ld}\|\Ld_1-\Ld_2\|\right)\delta^{-6}\langle\tau,\xi\rangle^3\label{fourier estimate 1}
			\end{align}
			where, $\mc{C}\subseteq\mc{H}$ is an open cone in $\mb R^{n}$.
		\end{lemma}
		
		\begin{proof}
			For a fixed nonzero $\xi\in\mc{H}$ and $k\in\{1,2,...,n\}$, we choose a set of $(n-1)$ linearly independent vectors from $\mc{C}_{\og_0}$ which are perpendicular to $\xi$ and denoted by $\{\og_i(\xi)\}_{1\le i\le n-1}$. Then, consider the following set of $n$ linear equations 
			\begin{align}
			&\ \ \  \sum_{j=1}^{n}{\omega}_i^j(\xi)\widehat{\eta_\delta^2\dd_kA^j}(\tau,\xi)=G_i(\tau,\xi), \qd i\in\{1,2,...,n-1\},\label{system of eqns 1}\\
			\tn{and}&\ \ \ \sum_{j=1}^{n}\xi_j\widehat{\eta_\delta^2\dd_kA^j}(\tau,\xi) =0,\  \left(\tn{since} \ \nabla_x\cdot A=0 \ \tn{gives}\  \nabla_x\cdot(\eta_\delta^2\ \dd_kA)=0\right).\label{system of eqns 2}
			\end{align}
			Here, $\{G_i(\tau, \xi)\}_{1\le i\le n-1}$ are real numbers with a upper bound  given by  \eqref{fourier estimate 1}. Now, consider the  matrix $M_\xi$ related to the system of equations \eqref{system of eqns 1} and \eqref{system of eqns 2} as follow
			\begin{align*}
			M_{\xi}=\begin{pmatrix}
			\og_1^1(\xi) & \og_1^2(\xi) &\cdots & \og_1^n(\xi) \\
			\og_2^1(\xi) & \og_2^2(\xi) &\cdots & \og_2^n(\xi) \\
			\vdots  & \vdots  &\ddots & \vdots  \\
			\og_{n-1}^1(\xi) & \og_{n-1}^2(\xi) &\cdots & \og_{n-1}^n(\xi) \\
			\f{\xi_1}{|\xi|} & \f{\xi_2}{|\xi|} &\cdots & \f{\xi_n}{|\xi|} 
			\end{pmatrix}. 	
			\end{align*}
			From our assumption of $\{\og_i(\xi)\}_{1\le i\le n-1}$, it is clear that $M_\xi$ is a non-singular matrix which is homogeneous of order zero in $\xi$. Hence, we take some $\xi_0\in\mb{S}^{n-1}\cap\mc{H}$ then, $0<|\det M_{\xi_0}|$. Now as we vary $\xi$ in a neighborhood of $\xi_0$ in $\mb{S}^{n-1}$ say $\widetilde{\mc{C}}$ we see the plane $\mc{H}_\xi$ changes continuously giving a linearly independent set of vectors $\{\og_i(\xi)\}_{1\le k\le n-1}$ which depend continuously on $\xi$. Then we have $c>0$ which is independent of $\xi\in\widetilde{\mc{C}}$ such that
			\begin{align}
			0<c\le|\det M_{\xi}|.\label{bound on determinat}
			\end{align}
			For any $r>0$ and $(\tau,\xi)\in\mb R\times\widetilde{\mc{C}}$,  using the uniform positive lower bound for the matrix $M_{\xi}$ in the system of linear equations \eqref{system of eqns 1} and \eqref{system of eqns 2}, we get 
			\begin{align}
			\lvert\widehat{\eta_\delta^2\dd_kA^j}(r\tau,r\xi)\rvert\le C\left(\f{1}{\rt\ld}+e^{\beta\ld}\|\Ld_1-\Ld_2\|\right)\delta^{-6}\langle r\tau,r\xi\rangle^3
			\end{align}
			where $k,j\in\{1,2,\cdots,n\}$.
			Define $\mc{C}\equiv\cup_{r>0}r\widetilde{\mc{C}},$ which is an open cone in $\mb R^{n}$. Thus, for $(\tau,\xi)\in\mb R\times\mc{C}$, we get \eqref{fourier estimate 1}. 
		\end{proof}
		
		\begin{lemma}\label{Conditional Stability} For $R\ge 1$ and $\delta\in(0,T/4)$, there exist  $C>0$ and $\theta\in(0,1)$ such that the following estimate holds
			\begin{align}
			\left\||\xi|\widehat{\eta_\delta^2 A}\right\|_{L^\infty(B(0,R))}\le Ce^{R(1-\theta)}\left(\frac{1}{\rt\lambda}+ e^{\beta\lambda}\|\Lambda_1-\Lambda_2\|\right)^{\theta}\delta^{-6\theta}R^{3\theta}.\label{vassella3}
			\end{align}
		\end{lemma}
		\begin{proof}
			Fix $k\in\{1,2,\cdots,n\}$ and consider the analytic function $f_{R,k}$ given by
			\[f_{R,k}(t,x)=\widehat{\eta_\delta^2\dd_{k}A}(Rt,Rx),\qd\m{ for }R>0\m{ and }(t,x)\in\mb R^{n+1}.\]
			For any multi-index $\gamma$ in $\mb{R}^{1+n}$ we observe 
			\begin{align*}      
			\qd\left|\partial_{(t,x)}^{\gamma}f_{R,k}(t,x)\right|&=\ \left|\partial_{(t,x)}^{\gamma}\widehat{\eta_\delta^2\dd_{k}A}(Rt,Rx)\right|=\left|\int_{\mathbb{R}^{1+n}}e^{-iR(s,y)\cdot (t,x)}(-i)^{|\gamma|}R^{|\gamma|}(s,y)^{\gamma}(\eta_\delta^2\dd_{k}A)(s,y)\ \d s\d y\right| \nn \\
			&\le \int_{\mathbb{R}^{1+n}}R^{|\gamma|}\left(s^2+|y|^2\right)^{\f{|\gamma|}{2}}(\eta_\delta^2\dd_{k}A)(s,y)\ \d s\d y,\nn\\
			&\le(2T^2)^{\f{|\gamma|}{2}}R^{|\gamma|}\int_{\mb R^{1+n}}|(\eta_\delta^2\dd_{k}A)(s,y)|\ \d s\d y. 
			\end{align*}
			Now using diam$(\Omega)<T$ and apriori bound of $A$, we have 
			\begin{align*}
			\qd\left|\partial_{(t,x)}^{\gamma}f_{R,k}(t,x)\right|\le C_*\ (2T^2)^{\f{|\gamma|}{2}}\ R^{|\gamma|}=C_*(\rt{2}T)^{|\gamma|}|\gamma|!\ \f{R^{|\gamma|}}{|\gamma|!},
			\end{align*} 
			which immediately gives
			\begin{align}
			\left|\partial_{(t,x)}^{\gamma}f_{R,k}(t,x)\right|\le C_*e^{R}\f{|\gamma|!}{(T^{-1})^{|\gamma|}},\ \m{ for }(t,x)\in\mb R^{n+1}\m{ and multi-index }\gamma. \label{prevessellla1}  
			\end{align}
			Since	$f_{R,k}$  satisfies \eqref{prevessellla1}, therefore using the  Vessella's conditional stability result \cite{Vessella} to $f_{R,k}$, we obtain  
			\begin{align}       
			\|f_{R,k}\|_{L^\infty(B(0,1))}\le Ce^{R(1-\theta)}\ \|f_{R,k}\|_{L^\infty((\mb{R}\times\mc{C})\cap B(0,1))}^\theta,\qd \m{for some}\ \theta\in(0,1). \label{vassella}
			\end{align}
			Since $\|f_{R,k}\|_{L^\infty(B(0,1))}=\left\|\widehat{\eta_\delta^2\dd_{k}A}\right\|_{L^\infty(B(0,R))}$, therefore using lemma \eqref{Inversion} and equation \eqref{vassella}, we get 
			\begin{align*}  
			\left\|\widehat{\eta_\delta^2\dd_{k}A}\right\|_{L^\infty(B(0,R))}\le Ce^{R(1-\theta)}\left(\frac{1}{\rt\lambda}+ e^{\beta\lambda}\|\Lambda_1-\Lambda_2\|\right)^{\theta}\delta^{-6\theta}(1+R^2)^{\f{3\theta}{2}}.
			\end{align*}  
			which can be expressed for $R\ge1$ in the following form
			\begin{align*}
			\left\||\xi|\widehat{\eta_\delta^2 A}\right\|_{L^\infty(B(0,R))}\le Ce^{R(1-\theta)}\left(\frac{1}{\rt\lambda}+ e^{\beta\lambda}\|\Lambda_1-\Lambda_2\|\right)^{\theta}\delta^{-6\theta}R^{3\theta}.
			\end{align*}        	
		\end{proof}
		
		\vspace*{-0.2cm}        
		Now combining \eqref{vassella3} and the apriori assumption on the potentials, we establish a Sobolev bound of $A$ in terms of the partial DN map. The main argument here is to set a comparison between the large parameters $\ld$ and $R$. Also, we have to choose the small parameter $\delta$ accordingly.
		But first we observe the following estimate
		\begin{align}    
		\left\|\eta_\delta^2 A\right\|_{L^2(Q)}^{\f{2}{\theta}}&=\left(\int_{\mb R^{1+n}}|\widehat{\eta_\delta^2 A}|^2(s,y)\ \d s\d y\right)^{\f{1}{\theta}}=\left(\int_{B(0,R)}|\widehat{\eta_\delta^2 A}|^2(s,y)\ \d s\d y+\int_{B(0,R)^c}|\widehat{\eta_\delta^2 A}|^2(s,y)\ \d s\d y\right)^{\f{1}{\theta}}\nn\\
		& \ \ \ \ \ \ \ \le \left(\underset{T_1}{\underbrace{\int_{B(0,R)}|\widehat{\eta_\delta^2 A}|^2(s,y)\ \d s\d y}}\right)^{\f{1}{\theta}}+\left(\underset{T_2}{\underbrace{\int_{B(0,R)^c}|\widehat{\eta_\delta^2 A}|^2(s,y)\ \d s\d y}}\right)^{\f{1}{\theta}}. \label{Sobolev norm 1}
		\end{align} 
		$T_2$ can be easily estimated after using the apriori assumptions for the potentials. We see
		\begin{align}\label{part 2}
		\begin{aligned}
		T_2&=\int_{B(0,R)^c}|\widehat{\eta_\delta^2 A}|^2(s,y)\ \d s\d y \\
		&\ \ \ \ \ \le \f{1}{R^2}\int_{\mb R^{1+n}}\langle\tau,\xi\rangle^2|\widehat{\eta_\delta^2 A}|^2(s,y)\ \d s\d y
		\le \ \f{1}{R^2}\|\eta_\delta^2 A\|_{H^{1}(Q)}^2 \le \ \f{C}{\delta^2R^2}.    
		\end{aligned}
		\end{align}
		To estimate $T_1$, we use lemma \ref{Conditional Stability} . We break $T_1$ into two parts. We consider $T_1=T_{11}+T_{12}$, where
		\begin{align}\label{part 1.a}
		\begin{aligned}
		T_{11}&:= 
		\int_{B(0,R)\cap\{(s,y);\ |y|\le R^{-\f{3}{n}}\}}|\widehat{\eta_\delta^2 A}|^2(s,y)\ \d s\d y\\
		& \ \ \ \ \ \ \ \ 	\le\  \|\widehat{\eta_\delta^2 A}\|_{L^\iy(\mb R^{n+1})}\int_{-R}^{R}\int_{|y|\le R^{-\f{3}{n}}}\ \d s\d y 
		\le \ CR^{-2}
		\end{aligned}
		\end{align}
		and
		\begin{align}\label{part 1.b}
		\begin{aligned}
		T_{12}&:=\int_{B(0,R)\cap\left\{(s,y);\ |y|> R^{-\f{3}{n}}\right\}}|\widehat{\eta_\delta^2 A}|^2(s,y)\ \d s\d y\\
		& \ \ \ \ \ \ \ \ \le e^{2R(1-\theta)}\left(\frac{1}{\rt\lambda}+ e^{\beta\lambda}\|\Lambda_1-\Lambda_2\|\right)^{2\theta}\delta^{-12\theta}R^{6\theta+n+1+\f{6}{n}}. 
		\end{aligned}
		\end{align}	
		Because of the support condition of $\{\eta_\delta\}_{\delta>0}$, we have
		\begin{align}
		\|A\|_{L^2(Q)}^2\le \|\eta_\delta^2A\|_{L^2(Q)}^2+\ C\delta. \label{relation}
		\end{align}
		Taking $\alpha=6+\f{n^2+n+6}{n\theta}$ and combining \eqref{Sobolev norm 1}-\eqref{part 1.b}, we obtain 
		\begin{align}
		\|A\|_{L^2(Q)}^{\f{2}{\theta}}\le& C\left(\f{R^{\alpha}}{\delta^{12}}e^{\f{2R(1-\theta)}{\theta}}\left(\frac{1}{\lambda}+e^{2\beta\lambda}\|\Lambda_1-\Lambda_2\|^2\right)+\f{1}{\delta^{\f{2}{\theta}}R^{\f{2}{\theta}}}+\delta^{\f{1}{\theta}}\right) \nn\\
		\le& C\left(\underset{I}{\underbrace{\f{R^{\alpha}e^{\f{2R(1-\theta)}{\theta}}}{\ld\delta^{12}}}}+\underset{II}{\underbrace{\f{R^{\alpha\theta}e^{\f{2R(1-\theta)}{\theta}+2\beta\ld}}{\delta^{12}}\|\Lambda_1-\Lambda_2\|^2}}+\underset{III}{\underbrace{\f{1}{\delta^{\f{2}{\theta}}R^{\f{2}{\theta}}}}}+\underset{IV}{\underbrace{\delta^{\f{1}{\theta}}}}\right).\label{main estimate 4}
		\end{align}
		We now choose $\ld,\delta$ and $R$ in a way so that the terms (I), (III) and (IV) in \eqref{main estimate 4} are comparable. That is when 
		\begin{align}
		\delta=\f{1}{R^{\f{2}{3}}}\ \ \tn{and}\ \  \ld=R^{\alpha+8+\f{2}{3\theta}}e^{\f{2R(1-\theta)}{\theta}}\label{delta,R and lambda 1}
		\end{align}
		and hence there exists $\kappa>0$ (independent of $R$) such that II of \eqref{main estimate 4} can be bounded by
		\begin{align}
		e^{e^{\kappa R}}\|\Lambda_1-\Lambda_2\|^2.\label{delta,R and lambda 2}
		\end{align}
		Combining \eqref{main estimate 4}-\eqref{delta,R and lambda 2}, it is clear that 
		\begin{align}
		\|A\|_{L^2(Q)}^{\f{2}{\theta}}\le C\left(\f{1}{R^{\f{2}{3\theta}}}+e^{e^{\kappa R}}\|\Lambda_1-\Lambda_2\|^2\right).\label{main estimate 5}
		\end{align}
		We now choose $R>0$ large enough (which in turn depends on the smallness of partial DN map) which is $R=\f{1}{\kappa}\log\big|\log\|\Lambda_1-\Lambda_2\|\big|$, so that \eqref{main estimate 5} becomes 
		\begin{align}\label{L^2 estimate for the vector potential}
		\|A\|_{L^2(Q)}^{\f{2}{\theta}}\le C\left(\|\Lambda_1-\Lambda_2\|+\left(\log|\log\|\Lambda_1-\Lambda_2\||\right)^{-\f{2}{3\theta}}\right).
		\end{align}
		We note that, \eqref{L^2 estimate for the vector potential} can be easily derived for the case when $\|\Ld_1-\Ld_2\|$ is not so small. This concludes the proof for stability of first order coefficients from the partial DN map.\\
		
		Now we establish the stability result for the zeroth order term. There will be no zeroth order term left in \eqref{integral identity} once we divide  it by large $\ld$. So we have to make necessary changes for deriving Fourier estimates. We explicitly use here the stability result for the first order terms \eqref{L^2 estimate for the vector potential}. We consider a different exponentially growing solutions for $\mc{L}_{A_2,q_2}$, whereas the geometric optics for $\mc{L}_{-A_1,\overline q_1}$ is same as before which is \eqref{GO 2}. We have
		\begin{align*}
		&u_{2}(t,x)=e^{\phi(t,x)}\left(B_{2}+R_{2}\right)(t,x)
		\end{align*}
		where  \begin{align*} B_{2}(t,x)=e^{-i\lb t\tau+x\cdot\xi\rb}\eta_\delta (t)e^{\left(\int_{0}^{\iy}\omega\cdot A_2(t,x+s\omega)\ \d s\right)}, 
		\end{align*}
		and $R_2\in L^2\left(0,T;H^2(\Og)\right)$ satisfying for $k\in\{0,1,2\}$
		\begin{align} \label{estimates of remainder 2}
		\|R_2\|_{L^2(0,T;H^k_\ld(\Omega))}\le C\ld^{-1+k}\delta^{-3}\langle\tau,\xi\rangle^3. 
		\end{align}
		For convenience, we rewrite the integral inequality 
		\begin{align}
		2\int_{Q}(A\cdot\nabla_xu_2)(t,x)\overline{v}(t,x)\ \d x\d t+\int_{Q}\widetilde{q}(t,x)u_2(t,x)\overline{v}(t,x)\ \d x\d t=\int_{\Sigma}\overline v(t,x)\dd_{\nu}u(t,x) \ \d S_{x}\d t.\qd\label{integral identity 2}
		\end{align}
		First, we simplify all the terms present in left hand side of \eqref{integral identity 2}. We observe 
		\begin{align}\label{estimte of q 1}
		\begin{aligned}
		\widetilde{q}(t,x)u_2(t,x)\overline{v}(t,x)=\widetilde{q}e^{-it\tau-ix\cdot \xi}\eta_\delta^2(t)e^{-\int_{0}^{\iy}\og\cdot A(t,x+s\og)\d s}
		+B_2(t,x)R(t,x)+B(t,x)R_2(t,x). 
		\end{aligned}
		\end{align}	
		We use Cauchy-Schwarz inequality alongwith \eqref{remainder terms} and \eqref{estimates of remainder 2} to obtain 
		\begin{align}
		    \|B_2R\|_{L^1(Q)}+\|BR_2\|_{L^1(Q)}\le \f{C}{\ld}\delta^{-3}\langle\tau,\xi\rangle^3
		\end{align}
		The other term present in the L.H.S of \eqref{integral identity 2} is
		\begin{align}
		    2(A\cdot\nabla_x u_2)(t,x)\overline v(t,x)=2\left(\ld\omega\cdot AB_2+\ld\omega\cdot AR_2+A\cdot\nabla_x B_2+A\cdot\nabla_x R_2\right)(t,x)\left(\overline{B}+\overline R\right)(t,x)
		\end{align}
		which is estimated by using the remainder term estimates given in \eqref{remainder terms} and \eqref{estimates of remainder 2}
		\begin{align}
		    \left|\int_{Q}(A\cdot\nabla_x u_2)(t,x)\overline v(t,x)\ \d x\d t\right|\le C\ld\|A\|_{L^2(Q)}\delta^{-3}\langle\tau,\xi\rangle^3.
		\end{align}
		To estimate the boundary term in the R.H.S of \eqref{integral identity 2}, we proceed as before and obtain
		\begin{align}
		    \left|\int_{\Sigma}\overline v(t,x)\dd_{\nu}u(t,x)\ \d S_x\d t\right|\le C\left(\rt{\f{\mc{K}}{\ld}}+\delta^{-6}\langle\tau,\xi\rangle^3e^{\beta\ld}\|\Lambda_1-\Lambda_2\|\right)
		\end{align}
		Here $\mc{K}$ is the R.H.S of the boundary Carleman estimate \eqref{Boundary Carleman estimate Theorem} applied to $\mc{L}_{A_1,q_1}$ for $u$,
		\begin{align}
		    \mc{K}=\int_{Q}e^{-2\phi}|\mc{L}_{A_1,q_1}u|^2\ \d x\d t+\ld\int_{\Sigma_{-}(\og)}e^{-2\phi}|\og\cdot\nu(x)||\dd_{\nu}u|^2\ \d S_x\d t.
		\end{align}
		From \eqref{PDE for the diff of soln 1} we have, $\mc{L}_{A_1,q_1}u=2A\cdot\nabla_x u_2+\tilde{q}u_2$. Hence we see
		\begin{align}
		    e^{-\phi}\mc{L}_{A_1,q_1}u=2\left(A\cdot\nabla_x B_2+\ld\og\cdot AR_2+A\cdot\nabla_x B_2+A\cdot\nabla_x R_2\right)+\tilde{q}(B_2+R_2).
		\end{align}
		Thus, we use \eqref{estimates of remainder 2} to obtain 
		\begin{align}\label{estimte of q 5}
		\mc{K}\le C\left(\ld^2\|A\|_{L^2(Q)}^2+1+e^{\beta\ld}\|\Ld_1-\Ld_2\|^2\right)\delta^{-12}\langle \tau,\xi\rangle^6. 
		\end{align}	
			
		Combining \eqref{integral identity 2}-\eqref{estimte of q 5}, we conclude for $(\tau,\xi)\in\mb R\times\mc{H}$
		\begin{align}\label{q estimate over cone}
		\begin{aligned}
		\left|\widehat{\eta_\delta^2 \widetilde{q}}(\tau,\xi)\right|&=\left\vert\int_{Q}e^{-i\lb t\tau+x\cdot\xi\rb}\eta_\delta^2 (t)\widetilde{q}(t,x)\ \d x\d t \right\vert \\
		&\ \ \ \ \  \le C\left(\ld\|A\|_{L^2(Q)}+\f{1}{\rt\ld}+e^{\beta\ld}\|\Ld_1-\Ld_2\|\right)\delta^{-6}\langle \tau,\xi\rangle^3.
		\end{aligned}
		\end{align}
		Basically \eqref{q estimate over cone} gives estimate for the Fourier transform of $\eta_\delta^2\widetilde{q}$ over the cone $\mb R\times\mc{H}$ in $\mb R^{1+n}$. So we apply Vessella's conditional stability result \cite{Vessella} as done before to obtain this estimate over arbitrary large balls. Mimicing the arguments presented before, we  get the following estimate similar to \eqref{main estimate 4}
		\begin{align}
		&\|\widetilde{q}\|_{L^2(Q)}^{\f{2}{\theta}}\le C\left(\f{R^{\alpha'}}{\delta^{12}}e^{\f{2R(1-\theta)}{\theta}}\left(\ld^2\|A\|_{L^2(Q)}^2+1+e^{\beta\ld}\|\Ld_1-\Ld_2\|^2\right)+\f{1}{\delta^{\f{2}{\theta}}R^{\f{2}{\theta}}}+\delta^{\f{1}{\theta}}\right) \nn\\
		&\  \le C\left(\underset{I}{\underbrace{\f{R^{\alpha'}e^{\f{2R(1-\theta)}{\theta}}}{\delta^{12}}\ld^2\|A\|_{L^2(Q)}^2}}+\underset{II}{\underbrace{\f{R^{\alpha'}e^{\f{2R(1-\theta)}{\theta}}}{\delta^{12}}}}+\underset{III}{\underbrace{\f{R^{\alpha\theta}e^{\f{2R(1-\theta)}{\theta}+\beta\ld}}{\delta^{12}}\|\Lambda_1-\Lambda_2\|^2}}+\underset{IV}{\underbrace{\f{1}{\delta^{\f{2}{\theta}}R^{\f{2}{\theta}}}}}+\underset{V}{\underbrace{\delta^{\f{1}{\theta}}}}\right)\label{pre estimate}
		\end{align}
		We choose $\delta$ and $R$ such that (II),(IV) and (V) of \eqref{pre estimate} are comparable. That is when 
		\begin{align*}
		\delta=\f{1}{R^{\f{2}{3}}}\ \tn{and}\  \ld=R^{\alpha'+8+\f{2}{3\theta}}e^{\f{2R(1-\theta)}{\theta}}
		\end{align*}
		here $\alpha'=6+\f{n+1}{\theta}$. 
		Now using the stability result \eqref{L^2 estimate for the vector potential}, we  obtain  
		\begin{align}\label{final estimate of q 1}
		\|\widetilde{q}\|_{L^2(Q)}^{\f{2}{\theta}}\le C\left({e^{\kappa R} }\|\Lambda_1-\Lambda_2\|^{2\mu_1}+{e^{\kappa R} }\left\vert\log|\log\|\Ld_1-\Ld_2\||\right\vert^{-2\mu_2}+e^{e^{\kappa R} }\|\Lambda_1-\Lambda_2\|^2+\f{1}{R^{\f{2}{3\theta}}}\right)
		\end{align}
		for some constants $\kappa>0$ and $\mu_1,\mu_2>0$.
		Taking $R=\f{\mu_1}{\kappa}\log\log|\log\|\Ld_1-\ld_2\||$, we get from \eqref{final estimate of q 1} that $\|\widetilde{q}\|_{L^2(Q)}^{\f{2}{\theta}}$ has the upper bound
		\begin{align}\label{final estimate of q 2}
		\begin{aligned}
		&\|\Lambda_1-\Lambda_2\|^{2\mu_1}\left(\log|\log\|\Ld_1-\Ld_2\||\right)^{-2\mu_2}+\left(\log|\log\|\Ld_1-\Ld_2\||\right)^{-\mu_2}\\
		&\ \ \ \ \ \ \ \ \ \ \ +\|\Lambda_1-\Lambda_2\|^{2}|\log\|\Ld_1-\Ld_2\|| 
		+\left(\log\log|\log\|\Ld_1-\Ld_2\||\right)^{-\f{2}{3\theta}}.
		\end{aligned}
		\end{align}
		We note that our choice of $R$ related to the smallness of $\delta$ and largeness of $\ld$. Hence, the estimate \eqref{final estimate of q 2} is valid only when $\|\Ld_1-\Ld_2\|$ is small enough. The other case follows easily. Also, we need smallness of $\|\Ld_1-\Ld_2\|$ such that the following hold 
		\begin{align*}
		\|\Lambda_1-\Lambda_2\|^{\mu_1}\left(\log|\log\|\Ld_1-\Ld_2\||\right)^{-2\mu_2}+\|\Lambda_1-\Lambda_2\||\log\|\Ld_1-\Ld_2\||\le C.
		\end{align*}
		Thus for both the cases, we arrive at the following estimate where $C,\alpha_1$ and $\alpha_2>0$
		\begin{align}\label{L^2 estimate for the scalar potential}
		\|\widetilde{q}\|_{L^2(Q)}\le C\left(\|\Ld_1-\Ld_2\|_{L^2(Q)}^{\alpha_1}+\left\vert\log\left|\log|\log\|\Ld_1-\ld_2\|\right|\right\vert^{-\alpha_2}\right).
		\end{align}
		Now we want to prove the stability result for $q:=q_1-q_2$. We recall 
		\begin{align*}
		q(t,x)=\widetilde{q}(t,x)+\nabla_x \cdot A(t,x)+\left(|A_1|^2-|A_2|^2\right)(t,x).
		\end{align*}
		Hence, we obtain the following 
		\begin{align}
		\|q\|_{L^2(Q)}\le \|\widetilde q\|_{L^2(Q)}+(2m+1)\|A\|_{L^2(0,T;H^1(\Og))}. \label{relation for q}
		\end{align}
		Since our assumptions on the first order perturbations are more than $H^1$, we can translate the $L^2$ norm estimates to that of $H^1$ using logarithmic convexity for Sobolev norms. Thus there exist $C>0$  and $\theta\in(0,1)$ depending only on $m$ and $Q$ so that we have 
		\begin{align}
		\|A\|_{H^1(Q)}\le C\|A\|_{L^2(Q)}^\theta \label{convexity argument}
		\end{align}
		Using  \eqref{convexity argument}, the $L^2$ stability results in \eqref{L^2 estimate for the vector potential} and  \eqref{L^2 estimate for the scalar potential} in \eqref{relation for q}, we  obtain 
		\begin{align*}
		\|q\|_{L^2(Q)}\le C\left(\|\Ld_1-\Ld_2\|^{\beta_1}+\left\vert\log|\log|\log\|\Ld_1-\Ld_2\||\right\vert^{-\beta_2}\right).
		\end{align*}
		for some $C,\beta_1$ and $\beta_2>0$. This completes the proof of Theorem \ref{Main Theorem}.

		\section*{Acknowledgments}	
		The authors would like to express their sincere gratitude and thanks to Venky Krishnan for many enlightening discussions and useful comments on this project.

	\end{document}